\documentclass[11pt]{article}

\usepackage{amssymb,latexsym,amsmath,amsthm,multicol,float,amstext, color}
\numberwithin{equation}{section}

\newcommand{\be}{\begin{eqnarray}}
\newcommand{\ee}{\end{eqnarray}}
\newcommand{\ce}{\begin{eqnarray*}}
\newcommand{\de}{\end{eqnarray*}}
\newtheorem{thm}{Theorem}[section]
\newtheorem{lemma}[thm]{Lemma}
\newtheorem{remark}[thm]{Remark}

\newtheorem{prop}[thm]{Proposition}
\newtheorem{exam}[thm]{Example}
\newtheorem{cor}[thm]{Corollary}

\newcommand{\F}{\mathcal{F}}
\newcommand{\B}{\mathcal{B}}
\newcommand{\Le}{\mathcal{L}}
\newcommand{\HH}{\mathcal{H}}
\newcommand{\D}{\mathcal{D}}

\newcommand{\ex}{\mathbb{E}}
\newcommand{\pr}{\mathbb{P}}

\newcommand{\R}{\mathbb{R}}
\newcommand{\EE}{\mathbb{E}}

\newcommand{\la}{\langle}
\newcommand{\ra}{\rangle}
\newcommand{\HP}{\mathcal{H}\mathcal{P}}
\newcommand{\I}{\mathbb{I}}
\newcommand{\varE}{\mathcal{E}}

\newcommand{\var}{\mbox{\rm Var}}
\begin{document}

%\title {Sharp Space-Time Regularity of the Solution to a
%Stochastic Heat Equation Driven by a Fractional-Colored Noise}
\title {Sharp Space-Time Regularity of the Solution to
Stochastic Heat Equation Driven by Fractional-Colored Noise}
\author{Randall Herrell, Renming Song, Dongsheng Wu, and Yimin Xiao
}

\maketitle

\begin{abstract}
In this paper, we study the following stochastic heat equation
\[
		\partial_tu=\Le u(t,x)+\dot{B},\quad
		u(0,x)=0,\quad 0\le t\le T,\quad x\in\R^d,
\]
where $\Le$ is the generator of a L\'evy process $X$ taking value in $\R^d$, 
$B$ is a fractional-colored Gaussian noise with Hurst index $H\in\left(\frac12,\,1\right)$ 
for the time variable and spatial covariance function $f$ which 
is the Fourier transform of a tempered measure $\mu.$

After establishing the existence of solution for the stochastic heat equation, 
we study the regularity of the solution $\{u(t,x),\, t\ge 0,\, x\in\R^d\}$ 
%in both time and space variables. Under mild conditions, the main results give 
in both time and space variables. Under mild conditions, we give
%the exact uniform modulus of continuity and the Chung-type laws of iterated 
the exact uniform modulus of continuity and a Chung-type law of iterated
logarithm for the sample function $(t,x)\mapsto u(t,x)$.  
Our results generalize and strengthen the corresponding
results of Balan and Tudor (2008) and Tudor and Xiao (2017).
\end{abstract}

{Running head}: Sharp space-time regularity of the solution to a stochastic heat equation

{\it 2000 AMS Classification numbers}: 60G15, 60J55, 60G18, 60F25.


\section{Introduction}
Stochastic partial differential equations (SPDE) driven by fractional Brownian motion (fBm) or other
fractional Gaussian noises have many applications in  biology, electrical engineering, finance, physics, among 
others, see, e.g., \cite{DMS03, BPS04, KouSinney04, CCL03}. The theoretical studies of SPDEs 
driven by fBm or other fractional Gaussian noises have been growing rapidly. We refer  
to, for example, \cite{BJQ15,balan,CHKN18, CHSS18, CHNT17,GLT06, HHNT15, HLN12, HN09, HNS11, 
maslowski-nualart03, nualart-vuillermont06, QS-tindel07, Song17,TTV} for recent developments.

In this paper, for a fixed constant $T>0,$ we consider the following stochastic heat equation 
\begin{equation}\label{Eq:she_l}
\partial_tu=\Le u(t,x)+\dot{B},\quad
u(0,x)=0,\quad 0\le t\le T,\quad x\in\R^d,
\end{equation}
where $\Le$ is the generator of a L\'evy process taking values in $\R^d$, and 
%where $B$ is a fractional-colored Gaussian noise with Hurst index $H\in\left(\frac12,\,1\right)$
 $B$ is a fractional-colored Gaussian noise with Hurst index $H\in\left(\frac12,\,1\right)$
in the time variable and spatial covariance function $f$ as in Balan and Tudor 
\cite{balan}. Namely, 
$$\left\{ B (t,A), t\in [0,T] , A\in \mathcal{B} (\mathbb{R}^{d}) \right\}, $$
%where $\mathcal{B} (\mathbb{R}^{d})$ denotes the family of Borel sets in $\mathbb{R}^{d}$,
is a centered Gaussian field  with covariance
%\begin{equation}
%\label{cov2}
$$
\EE \left( B (t, A) B (s, C) \right) = R_{H}(t,s) \int_{A}\int_{C} f(z-z') dzdz',
%\end{equation}
$$
where $R_{H}(t,s):=\frac{1}{2} (t^{2H}+ s^{2H} -\vert t-s \vert ^{2H} )$ is the covariance of
a fractional Brownian motion with index $H\in\left(\frac12,\,1\right)$, and $f$ is the Fourier transform of
a tempered measure $\mu$, which is defined by
$$
\int_{\mathbb{R}^d}f(x)\varphi(x)dx=\int_{\mathbb{R}^d}{\cal F}
%\varphi(\xi)\mu(d\xi), \quad \forall \varphi \in {\cal S}(\mathbb{R}^d), 
\varphi(\xi)\mu(d\xi), \quad \forall \varphi \in {\cal S}(\mathbb{R}^d).
$$
%where ${\cal S}(\mathbb{R}^d)$ denotes the Schwarz space on $\mathbb{R}^{d}$, and where 
Throughout this paper, $\mathcal{B} (\mathbb{R}^{d})$ denotes the family of Borel sets in $\mathbb{R}^{d}$, ${\cal S}(\mathbb{R}^d)$ denotes the Schwarz space on $\mathbb{R}^{d}$, and 
%${\cal F}\varphi$ denotes the Fourier transform of the function $\varphi$. Namely, 
${\cal F}\varphi$ denotes the Fourier transform of the function $\varphi$:
${\cal F}\varphi(\xi) = \int_{\R^d} e^{- i \langle \xi,  x\rangle } \varphi(x) dx.$ It is known that the 
mapping ${\cal F} : {\cal S}(\R^d) \to {\cal S}(\R^d)$ is an isomorphism which extends uniquely to a
unitary isomorphism of $L^2(\R^d).$

We will make use of the following %Parseval's
identity (cf. p.6 of \cite{Dalang99}): For any
$\varphi, \psi \in {\cal S}(\mathbb{R}^d)$,
\begin{equation}\label{parseval}
\int_{\mathbb{R}^d} \int_{\mathbb{R}^d} \varphi(x)f(x-y)\psi(y)dx dy=(2\pi) ^{-d}
\int_{\mathbb{R}^d}{\cal F} \varphi(\xi) \overline{{\cal F} \psi(\xi)}\mu(d\xi).
\end{equation}

%We point out here that, when $\Le=\Delta,$  the $d$-dimensional Laplacian operator, 
We point out here that, when $\Le=\Delta,$  the $d$-dimensional Laplacian, 
the existence of the solution 
of \eqref{Eq:she_l} has been studied by Balan and Tudor \cite{balan} and the regularity 
properties of the solution process $\{u(t, x), \, t\ge 0,\, x\in\R^d\}$ in the time variable $t$ %(while $x \in \R^d$ is fixed) 
(with $x \in \R^d$ fixed) 
%or the space variable $x$ (while $t > 0$ is fixed) has been studied by 
or the space variable $x$ (with $t > 0$ fixed) has been studied by 
Tudor and Xiao \cite{tudor}. 
The main objective of this paper is to study \eqref{Eq:she_l} for more general operator $\Le$ and to prove sharp 
%regularity properties of the solution in both time and space variables $(t, x)$ simultaneously.  
 regularity properties of the solution in time and space variables $(t, x)$ simultaneously.
Our results generalize and strengthen the corresponding results in \cite{balan, tudor} 
%to a class of broader stochastic heat equations. 
to a broader class of stochastic heat equations. 
Our approach is Fourier theoretical, and the main technical tool we use 
in this paper is the property of strong local nondeterminism of the Gaussian random field $\{u(t, x), \, t\ge 0,\, x\in\R^d\}$. 

We expect that, by developing approximation methods that are similar to those in \cite{KSXZ15,HP15} and by making use 
of the results in this paper, one can establish the exact uniform and local regularity results for the stochastic heat 
equation with multiplicative fractional Gaussian noises, particularly those that have been studied in \cite{BJQ15,CHKN18, 
CHNT17,CHSS18, HHNT15, HLN12, HN09, HNS11, Song17}.  We plan to pursue this line of research in a subsequent paper.

The rest of the paper is organized as follows. We first establish an existence result for the stochastic heat equation (\ref{Eq:she_l})
in Section \ref{Sec:Exi}, then study the regularity of the solution process under  mild conditions in Section \ref{Sec:Reg}. 
Finally, we provide a general result for a real valued centered Gaussian random field with stationary increments to be strongly 
locally nondeterministic in Section \ref{Sec:A}, 
%and we believe it has its own interest.
and we believe this general result is of independent interest.

Throughout this paper, for any appropriate measure $\mu$ on $\R^d$, we use $\F \mu(\xi)$
to denote the Fourier transform of $\mu$, that is
$$
\F\mu(\xi)=\int_{\R^d}e^{-i\langle \xi, x\rangle}\mu(dx), \qquad x\in \R^d.
$$

\noindent {\bf Acknowledgements}\, 
Research of Renming Song was supported in part by the Simons Foundation (\# 429343, Renming Song).
Research of Yimin Xiao was supported in part by grants from the National Science Foundation.

\section{Existence of the solution}\label{Sec:Exi}

Let $X=\{X_t,\,t\ge0\}$ be a L\'evy process taking values in $\R^d,$ with $X_0=0$ and characteristic exponent $\Psi(\xi)$ given by
\[
\EE\left(e^{i\la\xi, X_t\ra}\right)=e^{-t\Psi(\xi)},\qquad \forall \, t\ge0,\  \xi\in\R^d.
\]
Let $\Le$ be the generator of $X$. The domain of $\Le$ is given by
\[
{\rm Dom}(\Le)=\left\{\phi\in L^2\left(\R^d\right):\,\int_{\R^d}\left|\F^{-1} \phi(\xi)\right|^2\left|\Psi(\xi)\right|d\xi<\infty\right\},
\]
where $\F^{-1}$  denotes the inverse Fourier transform in $L^2\left(\R^d\right)$.

We first recall from \cite{balan} some facts about integration of deterministic functions with respect to the
%fractional-colored noise $B^H$. 
fractional-colored noise $B$.
Unless mentioned otherwise, we will use the same notation as in \cite{balan}.

Let $\D\left((0,\,T)\times\R^d\right)$ denote the space of all infinitely differentiable functions %whose support is compact and contained in $(0,\,T)\times\R^d,$ 
with compact support contained in $(0,\,T)\times\R^d$
and let $\HP$ be the completion of $\D\left((0,\,T)\times\R^d\right)$ with 
respect to the inner product
\begin{equation}\label{Eq:HP-ip}
\begin{split}
\la \varphi,\psi\ra_{\HP}
&=q_H\int_0^T\int_0^T\int_{\R^d}\int_{\R^d}\varphi(u,x)
|u-v|^{2H-2}f(x-y)\psi(v,y)dydxdudv\\
&=q_Hc_H\int_{\R}|\tau|^{1-2H} \int_{\R^d}\int_{\R^d}f(x-y)\F_{0,T} \varphi(\tau,x)
\overline{\F_{0,T} \psi(\tau,y)}dydxd\tau,
\end{split}
\end{equation}
where $q_H=H(2H-1),$ $c_H=\left[2^{2(1-H)}\sqrt{\pi}\right]^{-1}\Gamma(H-1/2)/\Gamma(1-H),$ and
$\F_{0,T} \varphi$ is the restricted Fourier transform of $\varphi$ in the variable $t\in(0,\,T)$ defined by
\[\F_{0,T}\varphi(\tau)=\int_0^Te^{-i\tau t}\varphi(t)dt.
\]
In (\ref{Eq:HP-ip}), the second equality follows from Lemma A.1.(b) in \cite{balan}. 
It follows from (\ref{Eq:HP-ip}) and (\ref{parseval}) that
\[
\begin{split}
\|\varphi\|_{\HP}&=q_Hc_H\int_{\R}|\tau|^{1-2H}d\tau \int_{\R^d}\int_{\R^d}f(x-y)\F_{0,T}\varphi(\tau,x)
\overline{\F_{0,T} \varphi(\tau,y)}dxdy\\
&=\frac{q_Hc_H }{(2 \pi)^d}\int_{\R}|\tau|^{1-2H}d\tau \int_{\R^d}\F(\F_{0,T}\varphi(\tau,\cdot))(\xi)
\F(\overline{\F_{0,T}\varphi(\tau,\cdot)(\xi)})\mu(d\xi).
\end{split}
\]

Let $B=\{B(\varphi):\,\varphi\in\D\left((0,\,T)\times\R^d\right)\}$ be a centered Gaussian process 
with covariance
\begin{equation}\label{e:iso-0}
\EE[B(\varphi)B(\psi)]=\la\varphi,\psi\ra_{\HP}.
\end{equation}
For any $t>0$ and $A\in\B(\R^d),$ one can define $B_t(A)=B\left(\I_{[0,\,t]\times A}\right)$ as 
the $L^2(\Omega)$-limit of the Cauchy sequence $\{B(\varphi_n)\},$ where $\{\varphi_n\}
\subset \D\left((0,\,T)\times\R^d\right)$ converges to $\I_{[0,\,t]\times A}$ pointwisely.
By a routine limiting argument, one can show that \eqref{e:iso-0} remains valid when $\varphi$ 
and $\psi$ are functions of the form $\I_{[0,\,t]\times A}$ with $t>0$ and $A\in\B(\R^d)$.

Let $\varE$ be the space of all linear combinations of indicator functions $\I_{[0,\,t]\times A},$ 
where $t\in[0,\,T],$ $A\in\B_b(\R^d)$ which is the class of all bounded Borel sets in $\R^d.$

One can extend the definition of $\EE[B(\varphi)B(\psi)]$ to $\varE$ by
linearity. Then we have
\begin{equation}\label{Eq:Iso}
\EE[B(\varphi)B(\psi)]=\la\varphi,\psi\ra_{\HP},\quad\forall\varphi,\,\psi\in\varE,
\end{equation}
i.e., $\varphi\to B(\varphi)$ is an isometry between $\left(\varE,\la\cdot,\cdot\ra_{\HP}\right)$ 
and $\HH^B,$ where $\HH^B$ is the Gaussian space generated by $\left\{B(\varphi),\,\varphi\in \D\left((0,\,T)\times\R^d\right)\right\}.$

Since the space $\HP$ is the completion of $\varE$ with respect to $\la\cdot,\cdot\ra_{\HP},$ 
the isometry (\ref{Eq:Iso}) can be extended to $\HP,$ giving us the stochastic integral 
of $\varphi\in\HP$ with respect to $B$. We denote this stochastic integral by
\[B(\varphi)=\int_0^T\int_{\R^d}\varphi(t,x)B(dtdx).\]

Now we study the existence of solution of (\ref{Eq:she_l}). Before doing this, we prove
some preliminary results first.

We assume that the L\'evy process $X=\{X_t\}$ has a transition density  which is given by
%\begin{equation}\label{Eq:den}
$$
p_t(x)=(2\pi)^{-d}\int_{\R^d}e^{-i\la x,\xi\ra}e^{-t\Psi(\xi)}d\xi=(2\pi)^{-d}\F e^{-t\Psi}(x),\quad\forall t>0.
%\end{equation}
$$
As in \cite[(3.25)]{balan}, we define the solution of the Cauchy problem \eqref{Eq:she_l} as follows. 
A random field $\{u(t, x): (t, x)\in [0, T]\times \R^d\}$ is said to be a solution of
\eqref{Eq:she_l} if for any $\eta\in \D\left((0,\,T)\times\R^d\right),$
%\begin{equation}\label{Eq:Exist}
$$
\int_0^T\int_{\R^d}u(t,x)\eta(t,x)dxdt=\int_0^T\int_{\R^d}\left(\eta*\tilde{p}\right)(t,x)B(dtdx),\quad a.s.,
%\end{equation}
$$
where $\tilde{p_s}(y)=p_s(-y).$

Denote $g_{t,x}(s,y)=p_{t-s}(x-y)\I_{\{s<t\}},\,s,\,t\in(0,\,T),\,x,\,y\in\R^d.$ We first derive conditions 
on the characteristic exponent $\Psi$ such that $\|g_{t,x}\|_{\HP}<\infty,$ which extends \cite[Theorem 3.12]{balan}.

Recall, by noting that
\[
(2\pi)^d\F^{-1} p_{t-s}(\cdot-x)(\xi)
=\EE\left[e^{i\la\xi, X_{t-s}-x\ra}\right]=e^{-i\la\xi, x\ra-(t-s)\Psi(\xi)},
\]
we have
%\begin{equation}\label{Eq:gtx1}
\begin{equation*}
\begin{split}
\|g_{t,x}\|_{\HP}&=q_H\int_0^t\int_0^t|s-r|^{2H-2}drds\int_{\R^d}\int_{\R^d}g_{t,x}(s,y)f(y-z)g_{t,x}(r,z)dydz\\
&=(2\pi)^{2d}q_H\int_0^t\int_0^t|s-r|^{2H-2}drds\int_{\R^d}
\F^{-1} g_{t,x}(s,\cdot)(\xi)\overline{\F^{-1} g_{t,x}(r,\cdot)(\xi)}\mu(d\xi)\\
&=(2\pi)^{2d}q_H\int_0^t\int_0^t|s-r|^{2H-2}drds\int_{\R^d}
\F^{-1} p_{t-s}(\cdot-x)(\xi)\overline{\F^{-1} p_{t-r}(\cdot-x)(\xi)}\mu(d\xi)\\
&=q_H\int_0^t\int_0^t|s-r|^{2H-2}drds\int_{\R^d}e^{-(t-s)\Psi(\xi)-(t-r)\Psi(-\xi)}\mu(d\xi).
\end{split}
%\end{equation}
\end{equation*}
Note that the display above says that $\|g_{t,x}\|_{\HP}$ is independent of $x$.
By Fubini's theorem, we have
%\begin{equation}\label{Eq:gtx2}
$$
\|g_{t,x}\|_{\HP}=q_H\int_{\R^d}\mu(\xi)\int_0^t\int_0^t|s-r|^{2H-2}e^{-(t-s)\Psi(\xi)-(t-r)\Psi(-\xi)}dsdr.
%\end{equation}
$$
We consider the inner integral (in $s$ and $r$) first. Define
%\begin{eqnarray*}
%h_1(s)&:=&\I_{[0,\,t]}(s)e^{-(t-s)\Psi(\xi)},\\
%h_2(r)&:=&\I_{[0,\,t]}(r)e^{-(t-r)\Psi(-\xi)},
%\end{eqnarray*}
$$
h_1(s):=\I_{[0,\,t]}(s)e^{-(t-s)\Psi(\xi)},\quad 
h_2(r):=\I_{[0,\,t]}(r)e^{-(t-r)\Psi(-\xi)}.
$$
%and recall Lemma A.1(b) from \cite{balan} that 
It follows from \cite[Lemma A.1(b)]{balan} that
for $\alpha\in(0,\,1),$ for every $\varphi,\,\psi$ from $L^2(a,\,b)$, we have 
\[\int_a^b\int_a^b\varphi(u)|u-v|^{-(1-\alpha)}\psi(v)dvdu=c_{\frac{\alpha+1}{2}}\int_\R |\tau|^{-\alpha}\F \varphi(\tau)\overline{\F \psi(\tau)}d\tau.\]
%We have, by choosing $\alpha=2H-1,\,\varphi=h_1,\,\psi=h_2,\,a=0,\,b=t,$ 
Thus, by choosing $\alpha=2H-1,\,\varphi=h_1,\,\psi=h_2,\,a=0,\,b=t,$ 
%\begin{equation}\label{Eq:gtx3}
\begin{equation*}
\begin{split}
&\int_0^t\int_0^t|s-r|^{2H-2}e^{-(t-s)\Psi(\xi)-(t-r)\Psi(-\xi)}dsdr\\
&=\int_{\R}\int_{\R}|s-r|^{2H-2}h_1(s)h_2(r)dsdr\\
&=c_H\int_{\R}|\tau|^{1-2H}\F h_1(\tau)\overline{\F h_2(\tau)}d\tau.
\end{split}
%\end{equation}
\end{equation*}
 By using the change of variables $t-s=s'$, we get
\begin{eqnarray*}
%\F h_1(\tau)&=&\int_0^te^{-i\tau s-(t-s)\Psi(\xi)}ds,\\
%&=& e^{-i\tau t}\int_0^te^{i\tau s-s\Psi(\xi)}ds\\
\F h_1(\tau)&=&\int_0^te^{-i\tau s-(t-s)\Psi(\xi)}ds 
=e^{-i\tau t}\int_0^te^{i\tau s-s\Psi(\xi)}ds\\
&=& \frac{-e^{-i\tau t}}{i\tau-\Psi(\xi)}\left(1-e^{i\tau t-t\Psi(\xi)}\right),
\end{eqnarray*}
and similarly,
\[\F h_2(\tau)=\frac{-e^{-i\tau t}}{i\tau-\Psi(-\xi)}\left(1-e^{i\tau t-t\Psi(-\xi)}\right).\]
Consequently, letting $K_H=q_Hc_H,$ we have
\begin{equation}\label{Eq:gtx4}
\begin{split}
&\|g_{t,x}\|_{\HP}\\
&=K_H\int_{\R^d}\mu(\xi)\int_{\R}|\tau|^{1-2H}\frac{1}{i\tau-\Psi(\xi)}\overline{\frac{1}{i\tau-\Psi(-\xi)}}\left(1-e^{i\tau t-t\Psi(\xi)}\right)
\overline{\left(1-e^{i\tau t-t\Psi(-\xi)}\right)}d\tau.
\end{split}
\end{equation}

For simplicity, we assume that $X$ is symmetric, which implies that $\Psi(\xi)=\Psi(-\xi).$ 
In this case, $\Psi(\xi)$ is real-valued, and (\ref{Eq:gtx4}) reduces to
\begin{equation}\label{Eq:gtx5}
\|g_{t,x}\|_{\HP}
=K_H\int_{\R^d}\mu(\xi)\int_{\R}\frac{|\tau|^{1-2H}}{\tau^2+\Psi(\xi)^2}\left|1-e^{i\tau t-t\Psi(\xi)}\right|^2d\tau,
\end{equation}
which allows us to prove the following theorem.

\begin{thm}\label{Thm:gtx}
Assume that $\Le$ in (\ref{Eq:she_l}) is the generator of a symmetric L\'evy process in $\R^d$ with characteristic 
exponent  $\Psi(\xi)$. For any $t>0$ fixed, if
%\begin{equation}\label{Eq:gtx6}
$$
\int_{\R^d}\frac{\mu(d\xi)}{t^{-2H}+\Psi(\xi)^{2H}}<\infty,
%\end{equation}
$$
then $\|g_{t,x}\|_{\HP}<\infty.$
\end{thm}
Before proving Theorem \ref{Thm:gtx}, we prove the following lemma first.

\begin{lemma}\label{Lem:gtx}
We have that
\begin{equation}\label{Eq:gtx7}
\int_{\R}|v|^{1-2H}\frac{1}{1+v^2}\left|1-e^{(iv-1)x}\right|^2dv\le Kx^{2H},\quad \forall x\in (0,\, 1),
\end{equation}
where $K$ is a positive finite constant.
\end{lemma}
\begin{proof}
Notice that the integral on the left hand side of (\ref{Eq:gtx7}) can be written as
%\begin{equation}\label{Eq:gtx8}
\begin{equation*}
\begin{split}
&\int_{\R}|v|^{1-2H}\frac{1-2e^{-x}\cos(vx)+e^{-2x}}{1+v^2}dv\\
&=\int_{|v|\le x^{-1}}|v|^{1-2H}\frac{1-2e^{-x}\cos(vx)+e^{-2x}}{1+v^2}dv\\
&\qquad\qquad\qquad+\int_{|v|>x^{-1}}|v|^{1-2H}\frac{1-2e^{-x}\cos(vx)+e^{-2x}}{1+v^2}dv\\
&:=I_1+I_2.
\end{split}
%\end{equation}
\end{equation*}
For $I_1,$ we use the inequality $\cos(vx)\ge1-(vx)^2$ for $|vx|\le 1$ to get that
\begin{equation}\label{Eq:gtx9}
\begin{split}
I_1&\le \int_{|v|\le x^{-1}}|v|^{1-2H}\frac{1-2e^{-x}(1-v^2x^2)+e^{-2x}}{1+v^2}dv\\
&\le \int_{|v|\le x^{-1}}\frac{|v|^{1-2H}}{1+v^2}(1-e^{-x})^2dv
+2\int_{|v|\le x^{-1}}|v|^{1-2H}x^2dv\\
&\le Kx^2+Kx^{2H}\le Kx^{2H}.
\end{split}
\end{equation}
For $I_2,$ we have
\begin{equation}\label{Eq:gtx10}
I_2\le 4\int_{|v|>x^{-1}}\frac{|v|^{1-2H}}{1+v^2}dv\le 8\int_{x^{-1}}^\infty v^{-(1+2H)}dv=Kx^{2H}.
\end{equation}
Combining (\ref{Eq:gtx9}) and (\ref{Eq:gtx10}), we arrive at the conclusion of Lemma \ref{Lem:gtx}.
\end{proof}

Now we are ready to prove Theorem \ref{Thm:gtx}.

\begin{proof}
We consider the inner integral with respect to $\tau$ in (\ref{Eq:gtx5}). Let $\tau=\Psi(\xi)v,$ we have
\begin{equation}\label{Eq:gtx11}
%\begin{split}
%& \int_{\R}\frac{|\tau|^{1-2H}}{\tau^2+\Psi(\xi)^2}\left|1-e^{i\tau t-t\Psi(\xi)}\right|^2d\tau\\
%&=\frac{1}{\Psi(\xi)^{2H}}\int_{\R}\frac{|v|^{1-2H}}{1+v^2}\left|1-e^{(iv-1)t\Psi(\xi)}\right|^2dv.
%\end{split}
\int_{\R}\frac{|\tau|^{1-2H}}{\tau^2+\Psi(\xi)^2}\left|1-e^{i\tau t-t\Psi(\xi)}\right|^2d\tau
=\frac{1}{\Psi(\xi)^{2H}}\int_{\R}\frac{|v|^{1-2H}}{1+v^2}\left|1-e^{(iv-1)t\Psi(\xi)}\right
|^2dv.
\end{equation}
Clearly, for all $\xi\in\R^d$,
\begin{equation}\label{Eq:gtx12}
\int_{\R}\frac{|v|^{1-2H}}{1+v^2}\left|1-e^{(iv-1)t\Psi(\xi)}\right|^2dv\le 8\int_0^\infty\frac{v^{1-2H}}{1+v^2}\le K<\infty
\end{equation}
for some positive constant $K$ independent of $\xi.$
To bound the integral on the right hand side of (\ref{Eq:gtx11}), we consider to two cases
separately:

\begin{description}
\item Case 1. If $t\Psi(\xi)>1,$ then by (\ref{Eq:gtx12}), we have
\begin{equation}\label{Eq:gtx13}
\frac{1}{\Psi(\xi)^{2H}}\int_{\R}\frac{|v|^{1-2H}}{1+v^2}\left|1-e^{(iv-1)t\Psi(\xi)}\right|^2dv\le \frac{K}{\Psi(\xi)^{2H}}.
\end{equation}
\item Case 2. If $t\Psi(\xi)\le 1,$ then by Lemma \ref{Lem:gtx}, we have
\begin{equation}\label{Eq:gtx14}
\frac{1}{\Psi(\xi)^{2H}}\int_{\R}\frac{|v|^{1-2H}}{1+v^2}\left|1-e^{(iv-1)t\Psi(\xi)}\right|^2dv\le Kt^{2H}.
\end{equation}
\end{description}
%Combining (\ref{Eq:gtx12}) and (\ref{Eq:gtx13}), we get
Combining (\ref{Eq:gtx12}), (\ref{Eq:gtx13}) and (\ref{Eq:gtx14}), we get
%\begin{equation}\label{Eq:gtx15}
\begin{equation*}
\begin{split}
&\frac{1}{\Psi(\xi)^{2H}}\int_{\R}\frac{|v|^{1-2H}}{1+v^2}\left|1-e^{(iv-1)t\Psi(\xi)}\right|^2dv\\
&\le K\left(\frac{1}{\Psi(\xi)^{2H}}\I_{\{t\Psi(\xi)>1\}}+t^{2H}\I_{\{t\Psi(\xi)\le 1\}}\right)\\
&\le\frac{2Kt^{2H}}{1+(t\Psi(\xi))^{2H}},
\end{split}
%\end{equation}
\end{equation*}
which proves Theorem \ref{Thm:gtx}.
\end{proof}

Theorem \ref{Thm:gtx} provides a sufficient condition for $\|g_{t,x}\|_{\HP}<\infty.$ Actually, this 
condition is also necessary provided $\Psi(\xi)$ satisfies an extra growth condition, i.e., the 
lower index of the L\'evy exponent $\beta^{''}>0,$ where $\beta^{''}$ is defined by  [cf. Blumenthal and Getoor (1961)]
\[\beta^{''}=\sup\{\gamma\ge0:\,\lim_{\|\xi\|\to\infty}\|\xi\|^{-\gamma}\rm{Re}\Psi(\xi)=\infty\}.\]

\begin{thm}\label{Thm:gtx1}
%Assume that the condition of Theorem \ref{Thm:gtx} holds. 
Assume that $\Le$ in (\ref{Eq:she_l}) is the generator of a symmetric L\'evy process in $\R^d$ with characteristic exponent  $\Psi(\xi)$.
If $\beta^{''}>0$ and
\begin{equation}\label{Eq:gtx16}
\int_{\R^d}\frac{\mu(d\xi)}{t^{-2H}+\Psi(\xi)^{2H}}=\infty,
\end{equation}
then $\|g_{t,x}\|_{\HP}=\infty.$
\end{thm}
\begin{proof}
%In this situation, there exists $K>0$ such that when $|\xi|\ge K$,
It follows from the assumption $\beta^{''}>0$ that there exists $K>0$ such that when $|\xi|\ge K$,
we have
\[\left|1-e^{(iv-1)t\Psi(\xi)}\right|\ge \frac12.\]
Therefore, by  (\ref{Eq:gtx11}),
%\begin{equation}\label{Eq:gtx17}
\begin{equation*}
\begin{split}
\|g_{t,x}\|_{\HP}
&=K_H\int_{\R^d}\mu(\xi)\int_{\R}\frac{|\tau|^{1-2H}}{\tau^2+\Psi(\xi)^2}\left|1-e^{i\tau t-t\Psi(\xi)}\right|^2d\tau\\
&\ge K_1\int_{|\xi|\ge K}\frac{\mu(d\xi)}{\Psi(\xi)^{2H}}\int_\R\frac{v^{1-2H}}{1+v^2}dv=\infty,
\end{split}
%\end{equation}
\end{equation*}
provided (\ref{Eq:gtx16}) holds.
\end{proof}

\begin{exam}\label{Ex:Sd_Riesz}
If $X$ is a symmetric L\'evy process in $\R^d$ with
\begin{equation}\label{e:stablelike}
\Psi(\xi)\asymp |\xi|^{\alpha}\quad{\rm for\,\, all }\,\, \xi\in\R^d \mbox{ with } |\xi|\ge 1,
\end{equation}
then for $\mu(d\xi)=|\xi|^{-\beta}d\xi$, $\|g_{t,x}\|_{\HP}<\infty$ is equivalent to
$$
\int_{\R^d}\frac{\mu(d\xi)}{1+\Psi(\xi)^{2H}}\asymp \int_{\R^d}\frac{d\xi}{|\xi|^\beta\left(1+|\xi|^{2\alpha H}\right)}<\infty,
$$
which in turn is equivalent to
$$
d-2\alpha H<\beta<d,\quad{\rm or}\quad H>\frac{d-\beta}{2\alpha}>0.
$$
Some concrete examples of symmetric L\'evy processes satisfying condition \eqref{e:stablelike}
are as follows:
\begin{description}
\item (i) Isotropic $\alpha$-stable process, for which $\Psi(\xi)=|\xi|^\alpha$;

\item (ii) relativistic $\alpha$-stable process with mass $m>0$ (\cite{CS, Ryznar}), for which
%$\Psi(\xi)\asymp|\xi|^\alpha$;
$\Psi(\xi)\asymp|\xi|^\alpha$ for all $\xi\in\R^d$ with $|\xi|\ge 1$;

\item (iii) the independent sum of an isotropic $\alpha$-sable process and an isotropic 
$\gamma$-stable process with $\gamma<\alpha$, for which $\Psi(\xi)=|\xi|^\alpha+|\xi|^\gamma$;

\item (iv) symmetric $\alpha$-stable process with L\'evy density comparable to that of the isotropic
$\alpha$-stable process;

\item (v) truncated $\alpha$-stable process (\cite{KS}), for which
$$
\Psi(\xi)=c\int_{|y|<1}\frac{1-\cos\langle \xi, y\rangle}{|y|^{d+\alpha}}dy
$$
for some constant $c>0$.
\end{description}
\end{exam}

The following result is an extension of \cite[Theorem 3.15]{balan}.

\begin{thm}\label{Thm:Exist}
Under the condition of Theorem \ref{Thm:gtx}, \eqref{Eq:she_l} has a solution
$\{u(t,x),\,(t,x)\in[0,\,T]\times\R^d\}$ and for all $(t,x)\in[0,\,T]\times\R^d$,
%\begin{equation}\label{Eq:Exist1}
$$
u(t,x)=\int_0^t\int_{\R^d}g_{t,x}(s,y)B(dsdy).
%\end{equation}
$$
\end{thm}
\begin{proof}
The argument is the same as that of \cite[Theorem 2.10]{balan}. We omit the details here.
\end{proof}

\section{Sharp regularity of the solution process in time and space}\label{Sec:Reg}

%In this section, we study regularities of the solution $u(t,x)$ of (\ref{Eq:she_l}) as a Gaussian 
%random field  in variables $(t,x)$.
Throughout this section, we weil assume that the conditions of Theorem \ref{Thm:gtx} hold. 
We study regularities of the solution $u(t,x)$ of (\ref{Eq:she_l}) as a Gaussian 
random field  in variables $(t,x)$.
%Similarly to the random string process in Mueller and Tribe \cite{MT02} 
Simimar to the case of the random string process in Mueller and Tribe \cite{MT02} 
(see also \cite{AX17,tudor}), we 
consider the Gaussian random fields $\{U(t,x),t\ge 0,\,\,x\in\R^d\}$ and $\{Y(t,x),t\ge 0,\,\,x\in\R^d\}$ defined, 
respectively, by
%\begin{equation}\label{Eq:U}
\begin{equation*}
\begin{split}
U(t,x) &= \int_{-\infty}^0\int_{\R^d}\left(p_{t-u}(x-y)-p_{-u}(x-y)\right)B(du,dy)\\
&\qquad+\int_0^t\int_{\R^d}p_{t-u}(x-y)B(du,dy)\\
&= \int_{\R}\int_{\R^d}\left(p_{(t-u)_+}(x-y)-p_{(-u)_+}(x-y)\right)B(du,dy),
\end{split}
%\end{equation}
\end{equation*}
and
%\begin{equation}\label{Eq:Y}
$$
Y(t,x) = \int_{-\infty}^0\int_{\R^d}\left(p_{t-u}(x-y)-p_{-u}(x-y)\right)B(du,dy),
%\end{equation}
$$
where $a_+=\max\{a,\,0\}$, thanks to the fact that $p_s(z)=0$ whenever $s<0.$

Clearly, $u(t,x)=U(t,x)-Y(t,x)$ for any $t\ge0$ and $x\in\R^d.$ In this section, we will, as in Tudor and Xiao (2017),
use this decomposition to obtain the exact uniform and local moduli of continuity of the solution $\{u(t,x),\,t\ge0\,\,
x\in\R^d\}.$ We would like to point out that, unlike in Tudor and Xiao (2017), where they studied the partial sample path 
%regularities of solution $u$ in time variable $t$ (when $x$ is fixed) 
regularities of solution $u$ in time variable $t$ (with $x$ fixed) 
%and in space variable $x$ (when $t$ is fixed), 
and in space variable $x$ (with $t$  fixed), 
we provide the corresponding results in time and space variables  $(t,x)$ simultaneously here. The key ingredient 
in our derivation is the strong local nondeterminism of $\{U(t,x),t\ge 0,\,\,x\in\R^d\}.$

We work on $\{Y(t,x)\}$ first. Denote $H_1=H-\frac{d-\beta}{\alpha}$ and $H_2=\alpha H_1.$

\begin{thm}\label{thm:Y}
Let $(t,x)\in I:=[a,\,b]\times [-M,\,M]^d$, where $[a,\,b]\subset[0,\,\infty)$ and $M>0$ 
%is fixed
are fixed.
\begin{description}
\item (i) If $a>0,$ then there is a modification of $\{Y(t,x)\}$ such that its sample function  
is almost surely continuously (partially) differentiable on $[a,\,b]\times [-M,\,M]^d.$
\item (ii) There is a finite positive constant $c$ such that
%\begin{equation}\label{Eq:Y1}
$$
\limsup_{|\varepsilon|\to0^+}\frac{\sup_{(t,x)\in I, (s,y)\in[0,\,\varepsilon]}|Y(t+s,x+y)-Y(t,x)|}
{\sqrt{\varphi(\varepsilon)\log(1+\varphi(\varepsilon)^{-1})}}\le c,
%\end{equation}
$$
where $\varphi(\varepsilon)=\varepsilon_1^{2H_1}+\sum_{j=2}^{d+1}\sigma(\varepsilon_j)$ for all 
$\varepsilon:=(\varepsilon_1,\,\varepsilon_2,\ldots,\varepsilon_{d+1})$ with $\sigma:\,\R_+\to\R_+$ defined by
\[\sigma(r)=\left\{\begin{array}{ccc}
           r^{2H_2} & \mbox{if }& 0<H_2<1, \\
           r^{2H_2}|\log r| & \mbox{if }& H_2=1,\\
           r^2 & \mbox{if }& H_2>1.
         \end{array} \right.
\]
\end{description}
\end{thm}

Theorem \ref{thm:Y} can be proved by using 
%the similar methods of Theorem 4.8 and Theorem 4.9 of 
arguments similar to that of the proofs of Theorems 4.8 and 4.9 of
Xue and Xiao (2011). We omit the details here.

Because of Theorem \ref{thm:Y}, the regularity properties of $\{u(t,x)\}$ are the same as that of $\{U(t,x)\}.$ 
Now we work on the Gaussian random field $U.$ Notice that, assuming $0<s<t$,
%\begin{equation}\label{Eq:U0}
$$
U(t,x)-U(s,y)=\int_{\R}\int_{\R^d}\left(p_{(t-u)_+}(x-z)-p_{(s-u)_+}(y-z)\right)B(du,dz).
%\end{equation}
$$

We have the following result:

\begin{thm}\label{thm:U}
The Gaussian random field $U=\{U(t,x),t\ge 0,\,x\in\R^d\}$ has stationary 
increments with spectral measure  given by
\[F_U(d\xi,d\tau)=\frac{1}{\tau^{2H-1}\left(\tau^2+\Psi(\xi)^2\right)}\mu(d\xi)d\tau.\]
\end{thm}
\begin{proof}
For $0\le s<t,$ by Parseval's identity (\ref{parseval}) we have
%\begin{equation}\label{Eq:U1}
\begin{equation*}
\begin{split}
&\ex(U(t,x)-U(s,y))^2\\
&= q_H\int_\R\int_\R |u-v|^{2H-2}dudv\int_{\R^d}\int_{\R^d}\left(p_{(t-u)_+}(x-z)-p_{(s-u)_+}(y-z)\right)\\
&\quad\quad\qquad\quad\qquad\quad\qquad\times f(z-z')\left(p_{(t-v)_+}(x-z')-p_{(s-v)_+}(y-z')\right)dzdz'\\
&= (2\pi)^{-d}q_H\int_\R\int_\R |u-v|^{2H-2}dudv\int_{\R^d}\F\left(p_{(t-u)_+}(x-\cdot)-p_{(s-u)_+}(y-\cdot)\right)(\xi)\\
&\quad\quad\qquad\quad\qquad\quad\qquad\times \overline{\F\left(p_{(t-v)_+}(x-\cdot)-p_{(s-v)_+}(y-\cdot)\right)(\xi)}\,\mu(d\xi),\\
\end{split}
%\end{equation}
\end{equation*}
%where $\omega_d$ is the surface area of the $d$-dimensional unit ball.

%Notice the fact that the Fourier transform of $p_t(x)$ is given by
Note that the Fourier transform of $p_t(x)$ is given by
%\begin{equation}\label{Eq:Fpt}
$$
\F p_t(x-\cdot)(\xi)=e^{i\la x,\xi\ra-t\Psi(\xi)}\I_{\{t>0\}},\quad\xi\in\R^d,
%\end{equation}
$$
%we derive
thus
%\begin{equation}\label{Eq:U2}
\begin{equation*}
\begin{split}
&\F\left(p_{(t-u)_+}(x-\cdot)-p_{(s-u)_+}(y-\cdot)\right)(\xi)\\
&=\F p_{(t-u)_+}(x-\cdot)(\xi)-\F p_{(s-u)_+}(y-\cdot)(\xi)\\
&=e^{-i\la x,\xi\ra}e^{-(t-u)\Psi(\xi)}\I_{\{t>u\}}-e^{-i\la y,\xi\ra}e^{-(s-u)\Psi(\xi)}\I_{\{s>u\}}\\
&:=\phi_{t,s}(u,\xi).
\end{split}
%\end{equation}
\end{equation*}
By applying Fubini's theorem and Parseval's identity (in $x$), we have
% \begin{equation}\label{Eq:U3}
 \begin{equation*}
 \begin{split}
 &\ex(U(t,x)-U(s,y))^2\\
 &=c_H\int_\R\int_\R |u-v|^{2H-2}dudv\int_{\R^d}\phi_{t,s}(u,\xi)\overline{\phi_{t,s}(v,\xi)}\,\mu(d\xi)\\
 &=c_H\int_{\R^d}\mu(d\xi)\int_{\R}\widehat{\phi_{t,s}}(\cdot,\xi)(\tau)\overline{\widehat{\phi_{t,s}}(\cdot,\xi)(\tau)}|\tau|^{1-2H}d\tau\\
 &=c_H\int_{\R^d}\mu(d\xi)\int_{\R}\left|\widehat{\phi_{t,s}}(\cdot,\xi)(\tau)\right|^2|\tau|^{1-2H}d\tau,
 \end{split}
% \end{equation}
 \end{equation*}
 where
% \begin{equation} \label{Eq:U4}
$$
 	\widehat{\phi_{t,s}}(\cdot,\xi)(\tau)=\int_{\R}e^{i\tau r}\phi_{t,s}(r,\xi)dr.
% \end{equation}
$$
% Again, notice that
By change of variables, we have
% \begin{equation} \label{Eq:U5}
 \begin{equation*}
 	\begin{split}
	\widehat{\phi_{t,s}}(\cdot,\xi)(\tau)&=e^{-i\la x,\xi\ra}\int_{-\infty}^te^{i\tau r}e^{-(t-r)\Psi(\xi)}dr-e^{-i\la y,\xi\ra}\int_{-\infty}^s e^{i\tau r}e^{-(s-r)\Psi(\xi)}dr\\
	&=e^{-i\la x,\xi\ra-i\tau t}\int_0^\infty	e^{-i\tau r-r\Psi(\xi)}e^{-i\la y,\xi\ra-i\tau s}\int_0^\infty	e^{-i\tau r-r\Psi(\xi)}dr\\
	&=\left(e^{-i(\la x,\xi\ra+\tau t)}-e^{-i(\la y,\xi\ra+\tau s)}\right)\frac{1}{i\tau+\Psi(\xi)},
 	\end{split}
% \end{equation}
 \end{equation*}
% where we have used variable substitutions in the above derivation, we obtain that
thus
% \begin{equation}\label{Eq:U6}
 \begin{equation*}
 \begin{split}
  &\ex(U(t,x)-U(s,y))^2\\
  &=c_H\int_{\R^d}\int_{\R}\left|e^{-i(\la x,\xi\ra+\tau t)}-e^{-i(\la y,\xi\ra+\tau s)}\right|^2\frac{|\tau|^{1-2H}}{\left|i\tau+\Psi(\xi)\right|^2}\mu(d\xi)d\tau\\
  &=2c_H\int_{\R^{d+1}}\left[1-\cos(\la x-y, \xi\ra+(t-s)\tau)\right]\frac{1}{|\tau|^{2H-1}\left(\tau^2+\Psi(\xi)^2\right)}\mu(d\xi)d\tau,
   \end{split}
% \end{equation}
 \end{equation*}
 which proves Theorem \ref{thm:U}.
\end{proof}

\begin{cor}\label{cor:FU1}
%	If $\mu$ satisfies the following condition that
	If $\mu$ satisfies the condition that
	\begin{equation}\label{Eq:Cond_mu}
	\mu(d\xi)\asymp |\xi|^{-\beta},\quad \mbox {with}\quad 0<\beta<d,
	\end{equation}
then the density function of $U$ satisfies
%\begin{equation}\label{Eq:FU_den}
$$
F_U(d\xi,d\tau)\asymp\frac{1}{\tau^{2H-1}\left(\tau^2+\Psi(\xi)^2\right)|\xi|^{\beta}}d\xi d\tau.
%\end{equation}
$$
If, in addition, we assume that 
\begin{equation}\label{Eq:Cond_psi}
\Psi(\xi)\asymp|\xi|^\alpha L(\xi),
\end{equation}
where $0<\alpha\le 2$ and $L(\cdot)$ is a slowly varying function at 
%$\infty.$ Then $U$ has a spectral measure 
$\infty$,  then $U$ has a spectral measure
which is ``comparable'' [in the sense of (19) in Tudor and Xiao (2017)] to
%\begin{equation}\label{Eq:FU_den1}
$$
F_U(d\xi,d\tau)\asymp\frac{1}{\tau^{2H-1}\left(\tau^2+|\xi|^{2\alpha}
	L(\xi)^2\right)|\xi|^{\beta}}d\xi d\tau.
%\end{equation}
$$
\end{cor}

\begin{remark}
Under the above two conditions (\ref{Eq:Cond_mu}) and (\ref{Eq:Cond_psi}), the stochastic heat equation (\ref{Eq:she_l}) has a solution if
%\begin{equation}\label{Eq:Cond_Sol}
$$
d-2\alpha H<\beta<d,\quad \mbox{or} \quad H>\frac{d-\beta}{2\alpha}\vee \frac 12.
%\end{equation}
$$
\end{remark}

For simplicity of presentation,  we assume from now on that $L(\cdot)\equiv 1.$
\begin{thm}\label{thm:SLNDU}
The Gaussian random field $U$ is strongly locally nondeterministic (SLND) in the following sense: 
there exists a positive constant $K$ such that for any positive 
integer $n$ and any $(t,x)\in(0,\,\infty)\times\R^d\backslash\{0\}$ and $(s^1,x^1),\ldots, (s^n,x^n)\in\R_+\times\R^d,$
\begin{equation}\label{Eq:SLNDU}
\var\left(U(t,x)\big | U(s^1,x^1),\ldots,U(s^n,x^n)\right)\ge K\min_{k=0,\ldots,n}(|t-s^k|^{H_1}+|x-x^k|^{H_2})^2.
\end{equation}
%Meanwhile,
Furthermore,
\begin{equation}\label{Eq:UpperU}
\ex\left[(U(t,x)-U(s,y))^2\right]\le K_1(|t-s|^{2H_1}+\sigma(|x-y|)),
\end{equation}
where $K_1>0$ is a constant, and 
%where $\sigma$ is defined the same as that in Theorem \ref{thm:Y}.
$\sigma$ is defined as  in Theorem \ref{thm:Y}.
In particular,
\begin{equation}\label{Eq:UpperU1}
\var\left(U(t,x)\big | U(s^1,x^1),\ldots,U(s^n,x^n)\right)\le K_1\min_{k=0,\ldots,n}(|t-s^k|^{2H_1}+\sigma(|x-x^k|)).
\end{equation}
\end{thm}
\begin{proof}
%By Theorem \ref{thm:SLND} in the Appendix, we only need to check that
It follows from Theorem \ref{thm:SLND} in the Appendix that, to prove (\ref{Eq:SLNDU}), 
we only need to check that
\begin{equation}
f_U(\tau,\xi):=\frac{1}{\tau^{2H-1}\left(\tau^2+|\xi|^{2\alpha}\right)|\xi|^{\beta}}
\end{equation}
satisfies (\ref{Eq:AScaling}) for some $\gamma=(\gamma_1,\,\gamma_2\ldots,\gamma_{d+1})\in(0,\,1)^{d+1}.$ 
In fact, by taking $\gamma_1=H_1,$ $\gamma_2=\cdots=\gamma_{d+1}=H_2=\alpha H_1,$ we have, for any $c>0,$
\begin{equation}\label{Eq:AsympS1}
f\left(c^{H_1^{-1}}\tau,c^{H_2^{-1}}\xi\right)=c^{-\left(\frac{2H+1}{H_1}+\frac{\beta}{H_2}\right)}f_U(\tau,\xi)=c^{-(2+Q)}f_U(\tau,\xi),
\end{equation}
where $Q=\frac1{H_1}+\frac{d}{H_2}.$ Therefore, (\ref{Eq:SLNDU}) follows 
%by applying Theorem \ref{thm:SLND}.
from Theorem \ref{thm:SLND}.

%Meanwhile, (\ref{Eq:UpperU}) follows from the proof of Lemma 3.2 in Xue and Xiao and 
The inequality (\ref{Eq:UpperU}) follows from the proofs of Lemma 3.2 in Xue and Xiao and 
Theorem 4 in Tudor and Xiao (2017). Finally, (\ref{Eq:UpperU1}) follows directly from (\ref{Eq:UpperU}). 
This completes the proof of Theorem \ref{thm:SLNDU}.
\end{proof}

Since $H_1\in(0,\,1),$ the solution is rough in $t.$ However, $H_2=\alpha H_1$ may be bigger than 1, 
and in this case $x\mapsto U(t,x)$ is differentiable. Hence, in order to give the exact uniform modulus 
of continuity, we distinguish three cases: (i) $H_2<1,$ (ii) $H_2=1$ and (iii) $H_2>1.$ We will study 
case (i) and case (iii) in this paper, case (ii) is more subtle and we have not been able to solve it completely.

We consider case (i) at first. We want to point out that when $0<\alpha\le 1,$ $H_2<1.$ In this case, as in 
\cite{tudor}, by applying the results on uniform and local moduli of continuity for Gaussian processes/fields 
(see, e.g. \cite{meers}), we have the corresponding regularity results on the solution $\{u(t,x),\,t\ge 0,\,x\in\R^d\},$ 
%[cf. Theorem 6.2 in Meerschaert, Wang and Xiao (2013)].
cf. Theorem 6.2 in Meerschaert, Wang and Xiao (2013).

\begin{prop}\label{Prop:mod}
Suppose that $H_2<1.$ Then the following results hold:
\begin{description}
\item (i) (Uniform modulus of continuity) For any $I:=[a,\,b]\times[-M,\,M]^d\subset\R_+\times\R^d$ with 
$0<a<b<\infty$ and $M>0,$ there is a constant $\kappa_1\in (0,\,\infty)$ such that
%\begin{equation}\label{Eq:Umod}
$$
\lim_{\varepsilon\to 0^+}\sup_{(t,x),(s,y)\in I:\rho(t,x;s,y)\le\varepsilon}\frac{|u(t,x)-u(s,y)|}
{\rho(t,x;s,y)\sqrt{\log(1+\rho(t,x;s,y)^{-1})}}=\kappa_1,\qquad a.s.
%\end{equation}
$$
where $\rho(t,x;s,y)=|t-s|^{H_1}+|x-y|^{H_2}.$
\item (ii) (Local modulus of continuity)  There is a constant $\kappa_2\in (0,\,\infty)$ 
such that for any $(t,x)\in I,$
%\begin{equation}\label{Eq:Lmod}
$$
\lim_{\varepsilon\to 0^+}\sup_{(s,y): \tilde{\rho}(s,y)\le\varepsilon}\frac{|u(t+s,x+y)-u(t,x)|}{\tilde{\rho} (s,y)
\sqrt{\log\log(1+\tilde{\rho}(s,y)^{-1})}}=\kappa_2,\qquad a.s.
%\end{equation}
$$
where $\tilde{\rho}(s,y)=|s|^{H_1}+|y|^{H_2}.$
\end{description}
\end{prop}

As corollaries, we have the corresponding results that generalize those in 
Tudor and Xiao (2017) for fixed $x$ and $t,$ respectively.
\begin{cor}\label{Cor:mod_t}
For $x\in \R^d$ fixed.  we have the following modulus of continuity results in time:
\begin{description}
\item (i) (Uniform modulus of continuity) For any $b>0,$ there is a constant $\kappa_3\in (0,\,\infty)$ such that
%\begin{equation}\label{Eq:Umod2}
$$
\lim_{\varepsilon\to 0^+}\sup_{t,s\in[0,b]:|t-s|\le\varepsilon}\frac{|u(t,x)-u(s,x)|}{|t-s|^{H_1}\sqrt{\log(1+|t-s|^{-1})}}=\kappa_3,\qquad a.s.
%\end{equation}
$$
\item (ii) (Local modulus of continuity)  There is a constant $\kappa_4\in (0,\,\infty)$ such that for any $t\in(0,\infty)$
%\begin{equation}\label{Eq:Lmod2}
$$
\lim_{\varepsilon\to 0^+}\sup_{|s| \le\varepsilon}\frac{|u(t+s,x)-u(t,x)|}{|s|^{H_1}\sqrt{\log\log(1+|s|^{-1})}}=\kappa_4,\qquad a.s.
%\end{equation}
$$
\end{description}
\end{cor}

\begin{cor}\label{Cor:mod_x}
Suppose that $H_2<1.$ For any $t>0$ fixed, we have the following modulus of continuity results in space:
\begin{description}
\item (i) (Uniform modulus of continuity) For any $M>0,$ there is a constant $\kappa_5\in (0,\,\infty)$ such that
%\begin{equation}\label{Eq:Umod3}
$$
\lim_{\varepsilon\to 0^+}\sup_{x,y\in [-M,M]^d: |x-y|\le\varepsilon}\frac{|u(t,x)-u(t,y)|}{|x-y|^{H_2}\sqrt{\log(1+|x-y|^{-1})}}=\kappa_5,\qquad a.s.
%\end{equation}
$$
\item (ii) (Local modulus of continuity) There is a constant $\kappa_6\in (0,\,\infty)$ such that for all $x\in\R^d$
%\begin{equation}\label{Eq:Lmod3}
$$
\lim_{\varepsilon\to 0^+}\sup_{y:|y|\le\varepsilon}\frac{|u(t,x+y)-u(t,x)|}{|y|^{H_2}\sqrt{\log\log(1+|y|^{-1})}}=\kappa_6,\qquad a.s.
%\end{equation}
$$
\end{description}
\end{cor}
Further properties on the local time and fractal behaviour of the solution $\{u(t,x)\},$ can also be derived from \cite{xiao97} \cite{xiao07} \cite{xiao09}.

Next, by combining Theorem \ref{thm:SLNDU} with Theorem 1.1 in Luan and Xiao \cite{LX10},  
we derive the following Chung-type 
%laws of the iterated logarithm for  the solution  $\{u(t,x)\}.$
law of the iterated logarithm for  the solution  $\{u(t,x)\}.$
\begin{prop}\label{Prop:Chung}
Suppose that $H_2<1.$ Then there is a constant $\kappa_7\in (0,\,\infty)$
such that for any $(t,x)\in I,$
%\begin{equation}\label{Eq:Chung}
$$
\liminf_{\varepsilon\to 0^+}\sup_{(s,y):\tilde{\rho}(s,y)\le\varepsilon}\frac{|u(t+s,x+y)-u(t,x)|}{\varepsilon 
(\log\log1/\varepsilon)^{-1/Q}}=\kappa_7,\qquad a.s.
%\end{equation}
$$
where $Q = \frac 1 {H_1}+ \frac{d}{H_2}.$
\end{prop}
\begin{proof} We only need to prove that
\begin{equation}\label{Eq:Chung1}
\liminf_{\varepsilon\to 0^+}\sup_{(s,y):\tilde{\rho}(s,y)\le\varepsilon}\frac{|U(t+s,x+y)-U(t,x)|}{\varepsilon 
	(\log\log1/\varepsilon)^{-1/Q}}=\kappa_7,\qquad a.s.
\end{equation}
Thanks to $H_2<1,$ we know $\sigma(r)=r^{2H_2}.$ By Theorem \ref{thm:SLNDU}, 
%we verify 
we see
that $U(t,x)$ satisfies Condition (C) in Luan and Xiao \cite{LX10}, and 
%we have that Eq. 
thus
(\ref{Eq:Chung1})	directly follows from their Theorem 1.1.  
\end{proof}

When $H_2>1,$  the solution process $\{u(t,x)\}$ has a version $\tilde{u}(t,x)$ such that
 $x\mapsto \tilde{u}(t,x)$ is continuously differentiable. More precisely, we now prove  the following result.
\begin{prop}\label{Prop:Smooth}
Suppose that $H_2>1.$ Then the solution process $\{u(t,x)\}$ has a version $\tilde{u}(t,x)$ 
with continuous sample functions such that $\frac{\partial \tilde{u}(t,x)}{\partial x_j}$ ($j=1,\ldots,d$) 
is continuous almost surely. Moreover, for any $M>0$,
there exists a positive positive random variable $K$ with all moments   such that for every $j = 1,  \ldots, d$,
the partial derivative $\frac{\partial}{\partial x_j} \tilde{u}(t,x)$ has the following modulus of continuity on $ [-M, M]^{d} $:
%\begin{equation}\label{mod-derivative}
$$
\sup_{x, y \in [-M, M]^{d}, |x-y|\le \varepsilon} \Big|\frac{\partial}{\partial x_j}\tilde{u}(t,x) - \frac{\partial}{\partial y_j} \tilde{u}(t,y) \Big|
\le K \varepsilon^{H_2-1} \sqrt{\log \frac 1 {\varepsilon}}.
%\end{equation}
$$
\end{prop}
The proof is similar to that of the proof of Theorem 4.8 in Xue and Xiao (2011) and the proof of Theorem 5 in 
Tudor and Xiao (2017). Therefore, we omit it here.

\section{Appendix: Strong local nondeterminism of a family of Gaussian random fields}
\label{Sec:A}

In this appendix, we prove the following general result on strong local nondeterminism for 
%a family of 
a class of
Gaussian random fields with stationary increments, which generalizes Theorem 3.2 in Xiao (2009) 
and may be of independent interest.
\begin{thm}\label{thm:SLND}
Let $\{X(t),\,t\in\R^N\}$ be a real-valued, centered Gaussian random field with stationary increments 
and spectral density $f(\lambda).$ If there exists a vector 
%$\gamma=(\gamma_1,\ldots,\gamma_N)\in(0,\,1)^N$ such that for all $c>0,$
$\gamma=(\gamma_1,\ldots,\gamma_N)\in(0,\,1)^N$ such that for all $a>0$
\begin{equation}\label{Eq:AScaling}
f\left(c^E\lambda\right)\asymp 
%c^{-(2+Q)}f(\lambda)\quad\forall\lambda\in\R^N\backslash\{0\},
a^{-(2+Q)}f(\lambda)\quad\forall\lambda\in\R^N\backslash\{0\},
\end{equation}
where $E$ is an $N\times N$ diagonal matrix with diagonal entries given by $\gamma_1^{-1},
\ldots,\gamma_N^{-1}$ and $Q=\sum_{j=1}^N\gamma_j^{-1}.$ Then, there exists a positive 
%constant $c_{_{\ref{Sec:A},1}}$ such that for any positive integer $n,$ and all $u,\,t^1,\ldots,t^n\in\R^N,$
constant $c$ such that for any positive integer $n,$ and all $u,\,t^1,\ldots,t^n\in\R^N,$
\begin{equation}\label{eq:slnd2}
\var\left(X(u)\big | X(t^1),\ldots,X(t^n)\right)\ge 
%c_{_{\ref{Sec:A},1}}\min_{k=0,\ldots,n}\rho(u,t^k)^2,
c\min_{k=0,\ldots,n}\rho(u,t^k)^2,
\end{equation}
where $t^0=0$ and $\rho(u,t):=\sum_{j=1}^N|u_j-t_j|^{\gamma_j}.$
\end{thm}

\begin{proof}
	Denote $ r \equiv \min\limits_{0 \le k \le n}\rho(u, t^k).$ Since
	the conditional variance in (\ref{eq:slnd2}) is the square of the
	$L^2(\pr)$-distance of $X(u)$ from the subspace generated by
	$\{X(t^1), \ldots, X(t^n)\}$, it is sufficient to prove that for all
	$a_k \in \R$ ($1 \le k \le n$),
	\begin{equation}\label{eq:slndl4}
	\EE \bigg(X(u)-\sum_{k=1}^n a_k\, X(t^k) \bigg)^2 \ge 
	%c_{_{\ref{Sec:A},1}}\,
	c\,
	r^{2}
	\end{equation}
	%where $c_{_{\ref{Sec:A},1}}>0$ is a constant which may only depend on $\gamma$ and $N$.
     where $c>0$ is a constant which depends only on $\gamma$ and $N$.
	
	By the stochastic integral representation [cf. (2.9) in Xiao (2009)] of $X$, and 
	thanks to the fact that $X$ has stationary increments,
	the left hand side of (\ref{eq:slndl4}) can be written as
	\begin{equation} \label{eq:slndl5}
	\EE \bigg(X(u)-\sum_{k=1}^n a_k X(t^k) \bigg)^2
	=\int_{\R^N}\bigg|e^{i\la u, \lambda \ra}-1 - \sum_{k=1}^{n} a_k\,
	\big(e^{i\la t^k, \,\lambda \ra}-1\big)\bigg|^2\, f(\lambda)\, d \lambda.
	\end{equation}
	Hence, we only need  to show
	\begin{equation}\label{Eq:slnd6}
	\begin{split}
	\int_{\R^N}\Big|e^{i\la u, \lambda \ra} - \sum_{k=0}^{n} a_k\, e^{i\la t^k,
		\,\lambda \ra } \Big|^2\, f(\lambda)\, d \lambda \ge c_{_{\ref{Sec:A},1}}\, r^{2},
	\end{split}
	\end{equation}
	where $t^0 = 0$ and $a_0 = -1 + \sum_{k=1}^n a_k$.
	
	Let $\delta(\cdot): \R^N \to [0, 1]$ be a function in $C^{\infty}
	(\R^N)$ such that $\delta (0)= 1$ and 
	%it vanishes outside the open
	$\delta$ vanishes outside the open
	ball $B_\rho(0, 1)$ in the metric $\rho$. Denote by $\widehat{\delta}$
	the Fourier transform of $\delta$. Then $\widehat{\delta} (\cdot)
	\in C^{\infty} (\R^N)$ as well and $\widehat{\delta} (\lambda)$ decays
	rapidly as $|\lambda| \to \infty$.
	
    Let $\delta_r(t)= r^{-Q}\delta(r^{-E} t)$. Then the
	inverse Fourier transform and a change of variables yield
	\begin{equation}\label{Eq:delta}
	\delta_r(t) = (2 \pi)^{-N}\int_{\R^N} e^{-i \la t, \lambda \ra}\,
	\widehat{\delta}(r^E \lambda)\, d \lambda.
	\end{equation}
	Since $\min\{\rho(u, t^k): 0 \le k \le n\} \ge r$, we have
	$\delta_r(u-t^k) = 0$ for $k=0, 1, \ldots, n$. This and
	(\ref{Eq:delta}) together imply that
	\begin{equation} \label{Eq:P1}
	\begin{split}
	J &:= \int_{\R^N} \bigg(e^{i\la u, \lambda \ra} - \sum_{k=0}^n a_k\, e^{i\la
		t^k, \lambda \ra} \bigg)\, e^{-i \la u, \lambda \ra}\,
	\widehat{\delta}(r^E \lambda)\, d \lambda  \\
	&  = (2 \pi)^N \bigg( \delta_r (0) - \sum_{k=0}^n a_k\,
	\delta_r (u - t^k) \bigg) \\
	& = (2 \pi)^N\, r^{-Q}.
	\end{split}
	\end{equation}
	On the other hand, by the Cauchy-Schwarz inequality, (\ref{Eq:AScaling}) and
	(\ref{eq:slndl5}), we have
	\begin{equation}\label{Eq:P2}
	\begin{split}
	J^2 &\le \int_{\R^N} \Big|e^{i\la u, \lambda\ra} - \sum_{k=0}^n a_k\,
	e^{i\la t^k, \lambda\ra} \Big|^2\, f(\lambda) \, d \lambda \cdot \int_{\R^N} \frac
	1  {f(\lambda)}\, \Big|
	\widehat{\delta}(r^E \lambda)\Big|^2\, d \lambda  \\
	&\le \EE\bigg( X(u) - \sum_{k=1}^n a_k X(t^k)\bigg)^2 \, \cdot
	r^{-Q}\, \int_{\R^N} \frac 1  {f( r^{-E}\, \lambda)}\, \Big|
	\widehat{\delta}(\lambda) \Big|^2\, d \lambda\\
	&\le c\, \EE\bigg( X(u) - \sum_{k=1}^n a_k X(t^k)\bigg)^2 \, \cdot
	r^{-2Q -2},
	\end{split}
	\end{equation}
	%where $c> 0$ is a constant which may only depend on $H$ and $N$.
	where $c> 0$ is a constant which only depends on $H$ and $N$.
	
	We square both sides of (\ref{Eq:P1}) and use (\ref{Eq:P2}) to
	obtain
	\[
	(2 \pi)^{2N}\, r^{-2Q} \le c\, r^{-2Q -2}\, \EE\bigg( X(u) -
	\sum_{k=1}^n a_k X(t^k)\bigg)^2.
	\]
	Hence (\ref{Eq:slnd6}) holds. This finishes the proof of the
	theorem.
\end{proof}

\begin{quote}
\begin{small}

\noindent \textsc{Randall Herrell}.
        Department of Mathematical Sciences,
        University of Alabama in Huntsville,
        Huntsville, AL 35899, U.S.A.\\
        E-mail: \texttt{rth0004@uah.edu}

\noindent \textsc{Renming Song}.
        Department of Mathematics,
        University of Illinois,
        Urbana, IL 61801, U.S.A.\\
        E-mail: \texttt{rsong@illinois.edu}\\
        URL: \texttt{http://www.math.illinois.edu/\~{}rsong}

 \noindent \textsc{Dongsheng Wu}.
        Department of Mathematical Sciences,
        University of Alabama in Huntsville,
        Huntsville, AL 35899, U.S.A.\\
        E-mail: \texttt{dongsheng.wu@uah.edu}\\
        URL: \texttt{http://webpages.uah.edu/\~{}dw0001}

 \noindent \textsc{Yimin Xiao}.
        Department of Statistics and Probability,
        Michigan State University,
        East Lansing, MI 48824, U.S.A.\\
        E-mail: \texttt{xiao@stt.msu.edu}\\
        URL: \texttt{www.stt.msu.edu/\~{}xiaoyimi}

\end{small}
\end{quote}
\end{document}